\documentclass[11pt]{article}

\usepackage{amsmath,amsfonts}
\usepackage{graphicx}
\usepackage{color}

\newif\ifblog
\newif\iftex
\blogfalse
\textrue

\usepackage{ulem}
\def\em{\it}
\def\emph#1{\textit{#1}}

\newcommand{\Trace}{\hbox{tr}}

\newcommand{\disp}{\displaystyle}
\newcommand{\cd}{}

\newcommand{\wdt}{\widetilde}
\newcommand{\e}{\varepsilon}

\newtheorem{theorem}{Theorem}
\newtheorem{lemma}[theorem]{Lemma}

\newtheorem{proposition}[theorem]{Proposition}
\newtheorem{example}{Example}
\newtheorem{remark}[theorem]{Remark}
\newenvironment{proof}{\noindent {\sc Proof:}}{$\Box$} 

\title{Weak Convergence Methods
for Approximation
of the Evaluation of Path-dependent Functionals\thanks{This research
was supported in part by the Research Grants Council of Hong Kong
No. CityU 103310, and in part by in part by the National Science Foundation
under DMS-1207667.}}
\author{Qingshuo Song,\thanks{Department of Mathematics, City
University of Hong Kong. 83 Tat Chee Avenue, Kowloon Tong, Hong
Kong, song.qingshuo@cityu.edu.hk.}
\and George  Yin,\thanks{Department of
Mathematics, Wayne State University, Detroit, Michigan 48202,
gyin@math.wayne.edu.}
\and \and
Qing Zhang\thanks{Department of Mathematics, The
University of Georgia, Athens, GA 30602, qingz@math.uga.edu.}}

\begin{document}

\maketitle

\begin{abstract}
In many applications, one needs to evaluate a path-dependent
 objective functional $V$ associated with
a continuous-time stochastic process $X$. 
Due to the
nonlinearity and
possible lack of Markovian property, more often than not, $V$ cannot be evaluated analytically,
and only Monte Carlo simulation or numerical approximation is possible.
In addition, such calculations often
   require the handling of stopping times, the usual dynamic programming approach may fall apart, and the
   continuity of the functional becomes main issue.
Denoting by $h$ the stepsize of the approximation sequence, 
this work develops a numerical scheme so that an approximating sequence of path-dependent functionals
$V^h$ converges to $V$. By a natural  division of labors,
the main task is divided into two parts. 
Given a sequence $X^h$ that converges weakly to $X$, 
 the first part provides sufficient conditions for the convergence of
the sequence of path-dependent functionals  $V^h$ to $V$. 
The second part 
constructs a sequence of approximations $X^h$ using Markov chain
approximation methods and 
demonstrates the weak convergence of  $X^h$ to $X$, when $X$ is
the solution of a stochastic differential equation. 
As a demonstration,  combining the
results of the two parts above, approximation of option pricing for
discrete-monitoring-barrier option underlying stochastic volatility model
is provided.
Different from the existing literature,
the weak convergence analysis is carried out by
using the Skorohod topology together with the continuous
mapping theorem. The advantage of this approach is that the
functional under study may be a function of stopping times, projection of
the underlying
diffusion on a sequence of random times, and/or maximum/minimum of the
underlying diffusion.



\end{abstract}

\newpage

\setlength{\baselineskip}{0.25in}
\section{Introduction and Examples} \label{sec:int}

In many applications, one needs to evaluate path-dependent objective functionals.
 They arise,
 for example,
  in derivative pricing, networked system analysis,
and Euler's approximation to solution of stochastic
differential equations. 
This paper is concerned with approximation methods for
computing such objective functions.
By a first glance, the problem may
appear as a standard approximation of
a stopping time problem
involving traditional
techniques.
Nevertheless,
a closer scrutiny  reveals that the problem under consideration
is far more challenging and difficult.
 The difficulties are in the following three aspects.
 \begin{itemize}
 \item[(i)] The path-dependence leads to fundamental difficulty in the evaluation of the underlying
  functionals with stopping times.
 \item[(ii)] The traditional dynamic programming approach falls apart,
not to mention any hope for a closed-form solution or
any viable numerical solutions for the associated
partial differential equations (PDEs).
Thus one naturally turns to
 approximation methods using Monte Carlo optimization.
To the best of our knowledge, convergence analysis is available
only in
few special cases due to the complexity of the nature of path dependence.
\item[(iii)] In evaluating the path-dependent functionals, operations involving max and min
are used. As a result, continuity becomes a major issue; more illustrations
and counter examples will be provided in Section \ref{sec:counter}. Standard weak convergence
argument does not apply.
\end{itemize}

To begin,
there are virtually no closed-form solutions 
involving path-dependent objective functionals.
One has to look for alternatives.
The next possibility is numerical approximation
using PDE-based
techniques.
Nevertheless, such possibility is ruled out due to the
 lack of
Markovian properties.
The only often used technique left is the Monte Carlo method.
While most of the existing methods for treating Monte Carlo
optimization are somewhat ad hoc,
this work develops a systematic alternative method for analyzing
the convergence of the approximation algorithm.
In this paper, we
use a weak convergence approach and
carry out the convergence analysis under the Skorohod topology.
One of the main ingredients is the use of a generalized projection operator.
With the use of such projections,
we proceed to explore the intrinsic properties of the Skorohod
space.
The convergence analysis developed in this paper
provides a
thorough understanding of the nature of path dependence and a general framework
for handling many path-dependent problems.

\subsection{Path-dependent Objective Functions}

We are interested in  approximating path-dependent functions in a finite-time horizon. For simplicity, we
focus our attention on the time interval $[0,1]$  to avoid using
more complex notation.
Let
 $C[0,1]$ be the collection of continuous real-valued functions defined on $[0,1]$,
 and $D[0,1]$ be the collection of all RCLL
(right continuous with left-hand limits)
functions on $[0,1]$. With ${\mathbb F}= \{\mathcal{F}_t: t\in [0,1]\}$
 satisfying the usual conditions and $\mathcal{F} = \mathcal{F}_{1}$, let
$(\Omega, \mathcal{F}, \mathbb{P}, \mathbb{F})$
be a filtered probability space;
see \cite{KS91}. As a convention,
we use capital letters
to denote random elements and use lower case letters
for
deterministic (or non-random) elements.
For instance, $x\cd\in D[0,1]$ is a RCLL
function
defined on $[0,1]$,
while $X\cd:\Omega \mapsto D[0,1]$ is a random RCLL process
with sample paths in
$D[0,1]$.
As usual, $X\cd$  can be regarded as a function
of two variables (time $t$ and sample point $\omega$).
That is, for each fixed $\omega \in \Omega$,
$X(\cdot,\omega)$ denotes a sample path and for each fixed $t$,
$X(t, \cdot)$ is a random element in $\mathbb R$.

\begin{figure}[htb]
  \centering
  \includegraphics[scale=0.35]{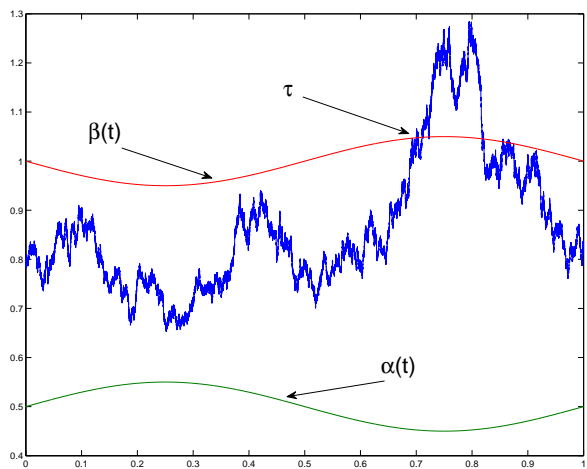}
  \caption{A demonstration of a sample path $X$, and  domain $\Gamma$
  bounded by lower barrier
  $\alpha$ and upper barrier $\beta$.}
  \label{fig:fig5}
\end{figure}

Throughout the paper, $\alpha\cd, \beta\cd \in C[0,1]$ are given
satisfying $\alpha(t) < \beta(t)$ for all $t\in [0,1]$.
The domain of our interest is defined by all the points bounded
by $\alpha\cd$ and $\beta\cd$, that is,
$$\Gamma = \{(a,t): \alpha(t) < a< \beta(t), t\in [0,1]\}.$$
The cross-section
$\Gamma(t) = (\alpha(t), \beta(t)) \subset \mathbb R$
may be time-dependent. We also denote its boundary by
$$\partial \Gamma = \{(\alpha(t), t):  t\in [0,1] \} \cup
\{(\beta(t), t): t\in [0,1] \}.$$
Consider an $\mathbb F$-adapted continuous process
$X: [0,1]\times \Omega \mapsto C[0,1]$ with initial $X(0) \in (\alpha(0), \beta(0))$.
Let
$\tau$ be the first hitting time
to the boundary of the time-dependence domain
\begin{equation}
  \label{eq:tau}
  \tau = \inf\{t>0: X(t) \notin \Gamma(t)\} \wedge 1.
\end{equation}
Now we introduce the notion of projection  operator (see \cite{Bil99})
$\Pi: D[0,1] \times [0,1]^{m} \mapsto \mathbb R^{m}$
as
\begin{equation}\label{eq:Pi}
 \Pi(x, \nu) = (x(\nu_{1}), \ldots , x(\nu_{m}))' \quad \forall x\in D[0,1], \nu
\in [0,1]^{m}.
\end{equation}

For  $x\cd\in D[0,1]$, we define a maximum process $x^*\cd$ by
$x^*(t) = \sup_{0\le s \le t} x(t)$. It is easy to check that $x^{*}\cd\in D[0,1]$.
For some
measurable function $g:\mathbb{R}^{4m+1} \to \mathbb{R}$,
we are interested in the computation of
the objective functional $V: C[0,1] \mapsto \mathbb R$
defined by, for given $\{\nu^{(i)}\in \mathbb [0,1]^{m}: i = 1, 2, 3, 4\}$
\begin{equation}
  \label{eq:V}
  V = \mathbb E [g ( \Pi(X,  \tau  \nu^{(1)} ), \Pi(X, \nu^{(2)}),
  \Pi(X^{*}, \tau \nu^{(3)}), \Pi(X^{*}, \nu^{(4)}), \tau)].
\end{equation}

Note that random times $\tau_{j} = \tau \nu^{(1)}_{j}$ are only measurable
with respect to $\mathcal F_{1}$, but may not be stopping times whenever
$\nu^{(1)}_{j}<1$. Moreover, the value function may be a measurable function depending on the initial states $V(x,t)$ if the functional $V$ is given
as a non-path-dependent
form of $ \mathbb E[f(X(T)]$ for some function $f:\mathbb R \mapsto \mathbb R$.
However, we emphasize that our main interest is
the computation when the functional is path-dependent in the form of
\eqref{eq:V}, which has a wide range of applications in mathematical finance.


\begin{remark} {\rm
Related literature in connection with a hitting time under a Markovian
framework can be found in Dufour and Piunovskiy \cite{DufourP},
de Saporta, et al. \cite{deDDG}, and Szpruch and Higham \cite{SzpruchH}
among others. In particular,  in \cite{DufourP}, the existence of
optimal stopping with constraints is established using
a convex analytic approach.
A numerical method for optimal stopping of
piecewise deterministic Markov processes is considered in
\cite{deDDG}. Bounds for the convergence rate of their algorithms
are obtained by introducing quantization of the post jump location
and path-adapted time discretization grids.
In \cite{SzpruchH}, an application in physics (thermodynamic limit)
involving mean hitting time behavior is considered.
}\end{remark}

\subsection{Examples}\label{sec:ex}
The evaluation of path-dependent objective functions
is of great interest in
many networked systems as well as in financial applications
such as the option pricing.
In particular, the value $V$ of \eqref{eq:V}
 can be considered as a general form of
a class of option prices including
look-back option, rebate/barrier option, Asian option, Bermuda option etc.
To illustrate,
consider the underlying stock price $X$ that is
a non-negative continuous martingale process under $\mathbb P$ that is
the corresponding {\it equivalent local martingale measure}. Throughout
this subsection, we  also assume that interest rate is  $r\ge 0$, and
$\nu^{(i)} = \frac 1 m (1, 2, \ldots, m)'$, $-\alpha(t) =  \infty$ and $\beta(t) = 1$. Thus, the first hitting time can be written by
$$\tau = \inf\{s>0: X(s) \notin (-\infty,1)\} \wedge 1.$$
We also assume the distribution of $X$ under $\mathbb P$ has no atom.
In particular, $\mathbb P \{X(t) = c\} = 0$ for all $c>0$ and $t>0$.

\begin{example}[Barrier option]
 \label{e:barrier}
 {\rm A barrier option is an option with a payoff depending on whether,
 within the life of the option, the price of the underlying asset reaches a
 specified level (the so-called  barrier); see more details in \cite{Mcd06}.
 There are three basic types of barrier
 options: Knock-out options, Knock-in options, and Rebate options.
 As an illustration, we consider up-and-in barrier call with strike $1/2$
   at maturity $T =1$, which is one special case of Knock-in options. Roughly
   speaking, up-and-in call activates (knocks-in) its payoff $(X(1) - 1/2)^{+}$
   if the upper barrier is touched before the expiration $T=1$.
   The precise price formula of the up-and-in call  is given by
   $$V = e^{-r} \mathbb{E}[(X(1) - \frac 1 2)^{+} I_{[0,1)}(\tau)] =
   e^{-r} \mathbb{E}[(X(1) - \frac 1 2)^{+} I_{[1,\infty)} (X^{*}(1))].$$
   In this case, the payoff function is
   $$g (x_{1}, x_{2}, \ldots, x_{4m+1}) = e^{-r} (x_{2m} - \frac 1 2)^{+} I_{[1,\infty)}(x_{4m}),$$
   which is
   discontinuous at $x_{4m} = 1$.
 }
\end{example}

\begin{example}[Discrete-monitoring-barrier option]
 \label{e:dbarrier}
{\rm
  In the related literature, most works
   assume continuous monitoring of the barrier like Example~\ref{e:barrier}. However,  in practice most barrier options traded in
   markets are monitored at discrete time. Unlike their continuous-time counterparts,
   there is essentially no closed form solution available, and even numerical
   pricing is difficult; see more related discussions in\cite{BGK99}
  and \cite{Kou03}.
   For instance,    the price formula of  up-and-in barrier call monitoring at discrete time instants
   $\{1/m, 2/m, \ldots, m-1/m, 1\}$
   with  strike $ 1 /2$ is
\begin{equation}\label{eq:db1}
 V = e^{-r}\mathbb{E}\Big[ \Big(X(1) - \frac 1 2\Big)^{+}
 I_{[1,\infty)}(\max_{1\le i \le m} X(i/m)) \Big].
\end{equation}
   The payoff function $g$ is written as
     $$g (x_{1}, x_{2}, \ldots, x_{4m+1}) = e^{-r} (x_{2m} - \frac 1 2)^{+} I_{[1,\infty)}(\max_{m+1\le i \le 2m} x_{i}).$$
   Note that $g\cd$ has linear growth, is unbounded, and is discontinuous.
   }
\end{example}

\subsection{Computational Methods}

To evaluate the path-dependent  objective functions,
a major difficulty arises from the lack of Markovian property.
This rules out the possibility of the conventional
numerical-PDE-based methods
including the finite difference or finite element methods.
In addition, the time-dependent barriers $\alpha$ and $\beta$
also create additional layer of the difficulty.
In this paper, we consider Monte Carlo methods
to obtain a feasible estimate of the path-dependent problems.
Monte Carlo method is a class of computational algorithms
that relies on some repeated random sampling to
evaluate its deterministic value using  its probabilistic fact.
This includes Euler-Maruyama approximation \cite{KP92} and
Markov chain approximation  \cite{Gla04}, \cite{Kus90}, and \cite{KD01}
among others.

The general idea of the Monte Carlo method in this vein is the following.
For each sample point $\omega \in \Omega$,
use $X^{h}(\cdot,\omega) \in D[0,1]$ to
denote a
simulated path for the underlying process
$X(\cdot,\omega)\in C[0,1]$
by a certain Monte Carlo method with
a small parameter $h$ (maybe a step size).   Define the approximated stopping time  by
$$
\tau^h =\inf\{t>0: X^h(t) \notin (\alpha(t), \beta(t))\} \wedge 1.$$
One can approximate $V$ in this way by computing
\begin{equation}
  \label{eq:Vh}
    V^{h} =  \mathbb E [g ( \Pi(X^{h},  \tau^{h}  \nu^{(1)} ), \Pi(X^{h}, \nu^{(2)}),
  \Pi(X^{h,*}, \tau^{h} \nu^{(3)}), \Pi(X^{h,*}, \nu^{(4)}), \tau^{h})].
\end{equation}
The main task of this approximating scheme is to design an appropriate
Monte Carlo method so that the desired convergence
takes place eventually, i.e.,
$$\lim_{h\to 0} V^h = V.$$

As
the objective value is given as an expectation
 of \eqref{eq:V},  $V$ is invariant under the same distribution,
and the usual requirements for the Monte Carlo method are intuitively
given as follows:
\begin{enumerate}
 \item [(H1)] $X^{h}\cd$ converges 
 weakly to $X\cd$
 as $h\to 0$, denoted by $X^{h} \cd\Rightarrow X\cd$.
 \item [(H2)] $g\cd$ is continuous.
 \end{enumerate}
A couple of natural questions
are as follows.
\begin{itemize}
 \item Are  (H1)-(H2)
 sufficient to guarantee the desired convergence
 $\lim_{h\to 0} V^{h} = V$?
 \item Can (H2) be possibly weakened to some discontinuous function $g\cd$
 so that the barrier option pricing (see Section~\ref{sec:ex})
 can be included?
\end{itemize}

\subsection{Counter Examples}\label{sec:counter}

Interestingly, there exist circumstances
that lead to counter examples in connection with
the desired convergence under (H1)-(H2).

\begin{figure}[htb]
  \centering
  \includegraphics[scale=0.35]{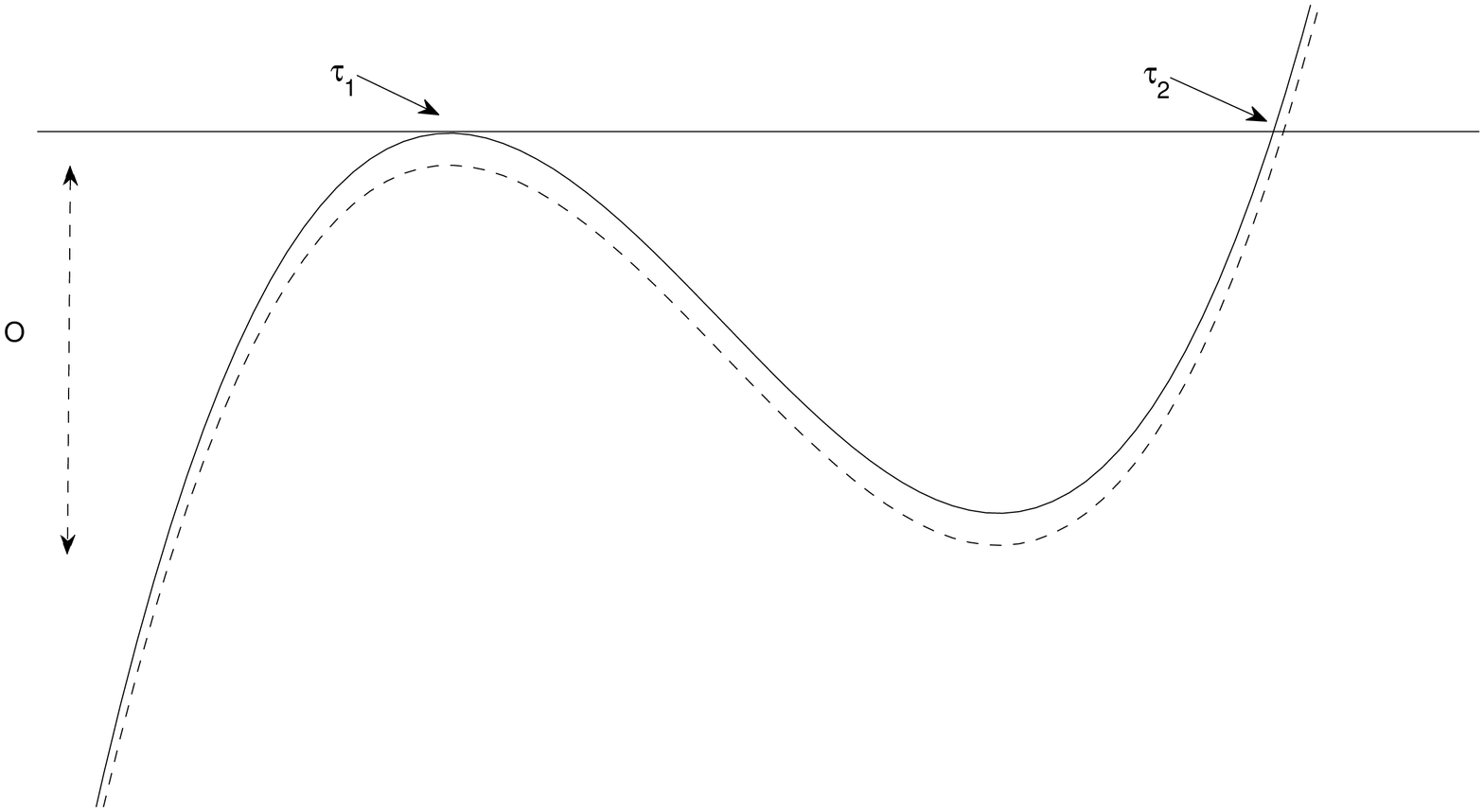}
  \caption{Demonstration of tangency problem: The sample paths of the
  two processes  (solid line and dotted line)
  are fairly close to each other, but their first exit times are far apart indicating a loss of continuity.}
  \label{fig:png}
\end{figure}


A noticeable counter example is
given in  Figure \ref{fig:png}, which is motivated
by the
so-called
{\it tangency problem}.
Consider two underlying processes $X^1\cd$ (solid
line) and $X^2\cd$ (dotted line) in Figure
\ref{fig:png}. No matter how close $X^1\cd$ and $X^2\cd$ be at the
initial states, the difference between their first exit time  $\tau_1$ and $\tau_2$ could
be far away;
see \cite{BSY10}. The above idea is illustrated in the
following example.

\begin{example} \label{ex:t}
 {\rm Let $\{X(s): 0\le s\le 1\}$ be a deterministic process given by
  $X(s) = 1 - (s-\frac 1 2)^{2}$. Let $\alpha({s}) = -\infty$ and
  $\beta_{s} = 1$ for all $s\ge 0$. Then the exit time of $X$ is
  $$\tau = \inf\{s>0: X(s)\notin (\alpha({s}), \beta({s}))\} \wedge 1
  = \frac 1 2.$$
 Define a family of processes parameterized by $h$ with
  $X^h(s)  = X(s) - h$. Although $X^{h}\cd$ converges to $X\cd$ in $L^{\infty}$
  as $h\to 0^{+}$, we note that
  $$\tau^{h} = \inf\{s>0: X^h(s) \notin (\alpha(s), \beta(s))  \} \wedge 1
  = 1$$
  not converging to $\tau = 1/2$.
}\end{example}

The next example gives a very different aspect from the first one
in the sense that
there are no barriers and the underlying process is 
a Bessel process. This system is often used in finance to model the
dynamics of asset prices, of the spot rate and of the stochastic volatility,
 or as a computational tool. In particular, computations for the celebrated
Cox-Ingersoll-Ross (CIR) and Constant Elasticity Variance (CEV) models can be carried out using Bessel processes.

\begin{example} \label{ex:c2}
{\rm
Let $X(t)$ be a Bessel process of 
order $3$ with initial $X(0) = 1$, i.e.,
\begin{equation}\label{eq:be3}
 X(t) = 1 + W(t) + \int_{0}^{t}\frac{ds}{X(s)},
\end{equation}
which is a well known strict
 local martingale with $\mathbb E[X(1)] <1$; see Page 336 of \cite{JYC09}.
A sample path of $X(t)$ is given in Figure~\ref{fig:Bessel}.
 Define $X^{h}(t) = X(t)\wedge \frac 1 h$. Fix a vector
 $\nu = (t_{1}\le \cdots \le t_{n})' \in [0,1]^{n}$, then  we have by definition
 $\Pi(X^{h}, \nu) \to \Pi(X, \nu)$ a.s., and the weak convergence
 $\Pi(X^{h}, \nu)\cd \Rightarrow \Pi(X, \nu)\cd$ takes place.
 Moreover, Kolmogorov consistency theorem together with
 the uniqueness of the weak solution of SDE \eqref{eq:be3}
 implies that
 $X^{h}\cd\Rightarrow X\cd$ due to
  the arbitrariness of $n$ and $\nu\in [0,1]^{n}$.
  However, since $X^{h}\cd$ is a bounded local martingale, it is a
  martingale. Thus we have
  $\lim_{h\to 0}\mathbb E[X^{h}(1)] = 1 > \mathbb E[X(1)].$

\begin{figure}[htb]
  \centering
  \includegraphics[scale=0.35]{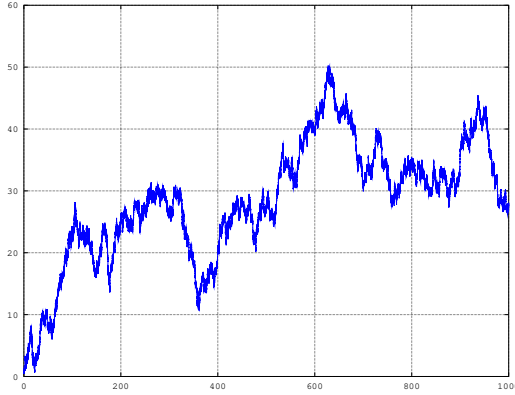}
  \caption{A sample path of Bessel process of order 3}
  \label{fig:Bessel}
\end{figure}
}\end{example}

\subsection{Goal and Outline of
the Paper}

The above examples show that
(H1)-(H2) may not be sufficient for Monte Carlo simulation to be
convergent to the right value of \eqref{eq:V}. This suggests the following
question:
\begin{itemize}
 \item[(Q1)] Given $X^{h}\cd \Rightarrow X\cd$, what are sufficient conditions
 to ensure
 the convergence $\lim_{h\to 0} V^{h} = V$? Is it possible to weaken
 the continuity of $g\cd$ to cover the barrier option?
\end{itemize}
In this paper, we first
present a rigorous proof of the convergence
in Section~\ref{sec:sur}.
The results in Theorem~\ref{thm1} provide
a set of sufficient conditions for the convergence.
It can also explain why Examples~\ref{ex:t} and \ref{ex:c2}
do not work.

Our approach is
based on the actual computations
using Skorohod metric in the space $D[0,1]$.
This is the first attempt in this context to the best of our knowledge.
Such an approach is  advantageous.
For example, the study of Monte Carlo convergence usually varies
with the particular form of the underlying payoff functions;
see \cite{Gla04} and the references therein. However,
Theorem~\ref{thm1} provides a
unified
theoretical basis for the convergence
in a general path-dependent
form with possibly discontinuous payoff function $g\cd$ and
finite time-dependent barrier $(\alpha(t), \beta(t))$. This covers
a wider range
of applications of Theorem~\ref{thm1} to various types
of options, for instance, barrier options, Asian options, look back options,
etc. It is also notable that the assumptions on the barrier in the path-dependent
problem have some similarities to
that
imposed for stochastic controlled exit problems; see \cite{BSY10}.
To complete the approximation of the value $V$ of \eqref{eq:V},
one should consider the following question in addition to (Q1).

\begin{itemize}
 \item [(Q2)] Given $X\cd$, how does one construct an
 approximating sequence of processes $X^{h}\cd$ such that $X^{h}\cd \Rightarrow X\cd$?
\end{itemize}
Among many of the possible answers to (Q2),  we will mainly
focus on the construction of $X^{h}\cd$
in the framework of Markov chain approximation in Section~\ref{sec:mar}.
Recall that from the book \cite{KD01},
a family of continuous approximating processes $X^{h}$
constructed using Markov chain is convergent to $X$ in distribution
if the Markov chain satisfies the local consistency
\cite[equation (9.4.2)]{KD01}; see also \cite[Theorem 10.4.1]{KD01}.
Compared to \cite{KD01}, we investigate the
weak convergence for more
generalized local consistency condition.
This enables us to cover various types of Monte Carlo simulations
in the framework of Markov chain approximation to verify its
weak convergence using the generalized local consistency. It essentially
extends the use of Markov chain approximation.

To proceed, the rest of the paper is arranged as follows. 
Section \ref{sec:sur} provides sufficient conditions for the convergence of the
approximating
path-dependent functional $V^{h}$ to the original functional $V$
given that there exists a sequence of random processes $X^{h}$
converging to the original process $X$ weakly.
Section \ref{sec:mar} addresses the question
 how to find a sequence of weak convergent $X^{h}$
 when the original process $X$ is the solution of a stochastic differential equation. 
 Since Section \ref{sec:mar} does not use any
result obtained in Section~\ref{sec:sur}, these two sections are independently accessible
to 
the reader.
Finally, the paper is concluded with
a demonstrating example on discrete-monitoring-barrier options together with 
 some further remarks in Section \ref{sec:fur}.

\section{Sufficient Conditions for Convergence} \label{sec:sur}
\subsection{Preliminaries }
We use the notation given in \cite{Bil99}.
Define a metric on $D[0,1]$ by
$$\|x-y\| = \sup_{t\in [0,1]} |x(t)-y(t)|, \
\forall x,y \in D[0,1].$$
Then, the continuous function space on $[0,1]$, denoted by  $C[0,1]$,
is complete
with respect to the uniform topology with the above metric $\|\cdot\|$.
On the other hand,  the RCLL function space defined on $[0,1]$, denoted by
$D[0,1]$, is equipped with the Skorohod
topology with the metric
$$\|x-y\|_s = \inf_{\lambda\in \Lambda}\{\|\lambda - \mathbb{I}\|, \|x \circ\lambda -y\|\}, \
\forall x,y\in D[0,1],$$
where
$x \circ y$ denotes the composite function of $x$ and $y$, and
$\Lambda$ is the collection of all continuous
increasing functions $\lambda\cd$
on $[0,1]$ with $\lambda(0) = 0$ and
$\lambda(1) = 1$, and $\mathbb I\in \Lambda$ is identity mapping. It is often useful to use the following fact:
$x_{n} \to x$ in Skorohod topology if and only if
    there exists $\lambda_n \in \Lambda$ such that
    \begin{equation}
      \label{eq:pf2}
      \|\lambda_n - \mathbb{I}\| \to 0, \quad
      \|x_n - x \circ \lambda_n\|
      \to 0,      \hbox{ as } n\to \infty.
    \end{equation}

 Note that $D[0,1]$ is not complete under the metric
$\|\cdot\|_{s}$, but it is complete under an equivalent metric
$\|\cdot\|^{o}_{s}$ defined by
$$\|x-y\|_s^{o} = \inf_{\lambda\in \Lambda} \Big\{ \sup_{s<t} \Big| \log \frac{\lambda (t) - \lambda (s)}{t-s} \Big|, \|x-y\| \Big\}, \
\forall x,y\in D[0,1].$$
For
notational convenience,
we will use $\|\cdot\|_{s}$ for the metric of $D[0,1]$
in the rest of the paper.
In particular, for a mapping $F:D[0,1] \mapsto \mathcal{M}$ with metric
space $(\mathcal M, \|\cdot \|_{\mathcal M})$, $F$ is said to be continuous
at some $x\in D[0,1]$ if
 $$\lim_{n\to \infty} \|F(x_{n}) - F(x)\|_{\mathcal{M}} = 0 \ \hbox{ whenever
 } \ \lim_{n\to \infty}\|x_{n} - x\|_{s} = 0.$$
In contrast, $x\in D[0,1]$ is said to be continuous at some $t\in (0,1)$, if
 $$\lim_{n\to \infty} |x(t_{n}) - x(t)| = 0 \ \hbox{ whenever } \ \lim_{n \to \infty}
 |t_{n} - t| = 0.$$

\subsection{Some Continuous Mappings Under Skorohod Topology}
In this section, we will discuss continuity of some useful mappings on $D[0,1]$ with respect to Skorohod topology.
Recall that, for $x\cd\in D[0,1]$, we define
$x^*\cd\in D[0,1]$ by $x^*(t) = \sup_{0\le s\le t}
x(s)$. We also define a mapping $\mathcal{M}: D[0,1] \mapsto D[0,1]$ as
  $$ \mathcal M(x) = x^{*}.$$
In addition, we are interested in the projection $\Pi$ in \eqref{eq:Pi} given by
$$ \Pi(x, \nu) = (x(\nu_{1}), \ldots , x(\nu_{m}))' \quad \forall x\in D[0,1], \nu
\in [0,1]^{m}$$
and a mapping $\pi\cd: D[0,1] \mapsto [0,1]$ defined by
\begin{equation}
 \label{eq:pi}
 \pi(x) = \inf\{t:x(t) \notin (\alpha(t), \beta(t))\}\wedge 1.
\end{equation}

\begin{lemma}
  \label{lem:w1} $\mathcal M$ is continuous on $D[0,1]$.
\end{lemma}

\begin{proof}
  Since $\lambda\cd$ is strictly increasing, $x^{*} \circ\lambda = (x\circ \lambda)^{*}$. Therefore,
  $$
  \begin{array}{ll}
    \|x^* - y^*\|_s &\displaystyle = \inf_{\lambda\in \Lambda}
    \{\|\lambda - \mathbb I\|, \sup_{t\in [0,1]}|x^*(\lambda(t)) -y^*(t)|\}
    \\ & \displaystyle
    = \inf_{\lambda\in \Lambda}
    \{\|\lambda - \mathbb I\|,
    \sup_{t\in [0,1]}|(x\circ \lambda)^*(t) -y^*(t)|\}
    \\ & \displaystyle
    \le \inf_{\lambda\in \Lambda}
    \{\|\lambda - \mathbb I\|,
    \sup_{t\in [0,1]} \sup_{0\le s\le t} |(x\circ \lambda)(s) -y(s)|\}\\
    & \displaystyle
    = \inf_{\lambda\in \Lambda}
    \{\|\lambda - \mathbb I\|,
    \|x\circ \lambda -y\|\} = \|x - y\|_s.
  \end{array}
  $$
  So, ${\mathcal M}$ is continuous in ${D[0,1]}$
  with respect to the Skorohod topology.
\end{proof}

\begin{lemma}\label{l:Pi}
 $\Pi$ is continuous at $(x, \nu) \in D[0,1] \times [0,1]^{m}$
 whenever $x$ is continuous at each $\nu_{i}$ of $i = 1, 2, \ldots, m$.
\end{lemma}
\begin{proof}
 Given a sequence of vectors $\nu^{(n)} \to \nu$, we observe that
 $$|\Pi(x, \nu^{(n)}) - \Pi(x, \nu)| \le \sum_{i = 1}^{m} |x(\nu_{i}^{(n)}) - x(\nu_{i})|,$$ and the continuity of $\Pi$ in the variable $\nu$ follows directly from the continuity of
 $x$ at $\nu_{i}$ for each $i$. Next, we show the continuity of $\Pi$ in  the variable $x$. To proceed, we fix $\nu$, and an arbitrary sequence $\{x_{n}\}$ satisfying $\lim_{n}\|x_{n}-  x\|_{s} = 0$.  Then, there exists  $\{\lambda_{n}\}$ satisfying \eqref{eq:pf2}.  Therefore, we have
    $$
    \begin{array}{ll}
     |\Pi (x_{n}, \nu) - \Pi (x, \nu)| &
    \disp =  \sum_{i=1}^{m} |x_n(\nu_{i}) - x(\nu_{i})| \\
     & \disp \le
     \sum_{i=1}^{m} (
           |x_n(\nu_{i}) - x (\lambda_n (\nu_{i}))| + |x (\lambda_n(\nu_{i})) - x(\nu_{i})| )\\
      &\disp \le \|x_n - x\circ \lambda_n)\| + \sum_{i=1}^{m}
      |x(\lambda_n(\nu_{i})) - x(\nu_{i})|
    \end{array}
    $$
   Note that the first term $\|x_n - x\circ \lambda_n\|\to 0$
    as $n\to \infty$ by \eqref{eq:pf2}.
    On the other hand, since $\|\lambda_n -\mathbb I\| \to 0$ as $n\to \infty$
    by \eqref{eq:pf2},
    we have $|\lambda_n(\nu_{i}) - \nu_{i}| \to 0$ as $n\to \infty$.
    Hence, the second term
    $\sum_{i=1}^{m}
      |x(\lambda_n(\nu_{i})) - x(\nu_{i})|$ by continuity
    of $x$ at $\nu_{i}$.
    Thus, we conclude the continuity of $\Pi$ in  the variable $x$.
\end{proof}

The continuity of $\pi$ of \eqref{eq:pi} is a rather tricky part.
To proceed, let us partition the space $C[0,1]$ as follows: Define
 $$C_1 = \{x\in C[0,1]: \pi(x) <1, \ x(\pi(x)) = \beta({\pi(x)}), \
    \inf\{t>\pi(x): x(t)>\beta(t)\} = \pi(x) \},$$
$$C_2 = \{x\in C[0,1]: \pi(x) <1, \ x(\pi(x)) = \alpha({\pi(x)}),
      \ \inf\{t>\pi(x): x(t)< \alpha(t) \} = \pi(x) \},$$
     and
$$C_3 = \{x\in C[0,1]: \pi(x) =1 \}, \hbox{ and }
      C_{4} = C[0,1] \setminus (\cup_{i=1}^{3} C_{i}).$$
 Accordingly, we can write
$C[0,1] = \cup_{i=1}^{4} C_{i}$ and $C_{i} \cap C_{j} = \emptyset$ for
$i\neq j$. As for the illustration, one can see that
the four curves depicted in Figure~\ref{fig:png1} belong to
four different subsets separately, that is, $L_{i} \in C_{i}$ for $i = 1,2,3,4$.

\begin{figure}[htb]
  \centering
  \includegraphics[scale=0.35]{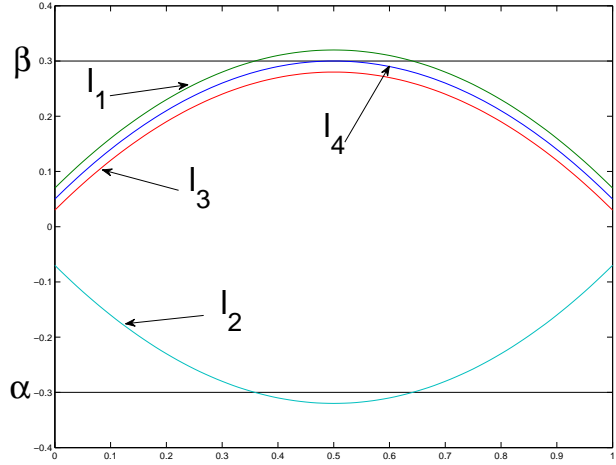}
  \caption{Illustration for the partition of $C[0,1]$}
  \label{fig:png1}
\end{figure}

\begin{lemma}
 \label{l:pi} The
 $\pi$ defined in \eqref{eq:pi} is continuous at each $x\in \cup_{i=1}^3 C_i$.
\end{lemma}
\begin{proof}
First, $\pi\cd$ is continuous at $x\in C_3$  thanks to Lemma~\ref{lem:w1}.
     We assume $-\alpha(t) + \beta(t) < \infty$ for each $t$ without loss of generality.
   Fix $x\in C_{1}$, and $\{x_{n}\}$ be
   an arbitrary sequence in $D[0,1]$ satisfying
   $\|x_n - x\|_s \to 0$ as $n\to \infty$.
  By the definition of $\|\cdot\|_{s}$,
  it implies that there exists $\lambda_n \in \Lambda$ such that
  \begin{equation}
    \label{eq:pf3}
    \|\lambda_n - \mathbb I\| \to 0, \quad
    \|x_n \circ \lambda_n - x\| \to 0, \hbox{ as } n\to \infty.
  \end{equation}

  \begin{itemize}
  \item Case 1.
    Define, for any $\varepsilon >0$
    $$\delta_{\varepsilon} :=
     \frac 1 2 \sup_{\pi(x)<t<\pi(x)+\varepsilon} (x(t) - \beta(t)).$$
    By the definition of $C_{1}$, we have
    $\delta_{\varepsilon}>0$ for all $\varepsilon>0$.
    Moreover, by continuity of $x$,
    there exists  a time
    $t_{\varepsilon}\in (\pi(x), \pi(x) + \varepsilon)$ such that
    $x(t_{\varepsilon}) > \beta(t_{\varepsilon}) + \delta_{\varepsilon}$.
    Next,  we observe that
    $$
    |x_n (\lambda_n (t_{\varepsilon})) - x(t_{\varepsilon})|
    \le \|x_n \circ \lambda_n - x\| \to 0,
    \ \hbox{ as } n\to \infty.$$
    This implies that for some large $N_{\varepsilon, 1}$
    \begin{equation}\label{eq:pf4}
    x_n(\lambda_n(t_{\varepsilon})) \ge x(t_{\varepsilon}) -
    \frac {\delta_{\varepsilon}} 2
    > \beta(t_{\varepsilon}) + \frac {\delta_{\varepsilon}} 2, \
    \forall n\ge N_{\varepsilon, 1}.
    \end{equation}

    On the other hand, there exists $r_{\varepsilon}>0$ due to the
    continuity of $\beta$ such that
    $$\sup_{t_{\varepsilon} - r_{\varepsilon}<t < t_{\varepsilon} +
    r_{\varepsilon}} |\beta(t) - \beta(t_{\varepsilon})| <
    \frac {\delta_{\varepsilon}}{4}.$$
    Moreover, since $\|\lambda_n - \mathbb I\| \to 0$, we have for some large
    enough $N_{\varepsilon,2}$,
    $$\|\lambda_{n} - \mathbb I\|<r_{\varepsilon}, \ \forall n \ge N_{\varepsilon,2}.$$
    Therefore, it leads to
    \begin{equation}\label{eq:pf5}
    |\beta (\lambda_{n}(t_{\varepsilon})) - \beta(t_{\varepsilon})|
    \le \sup_{t_{\varepsilon} - r_{\varepsilon}<t < t_{\varepsilon} +
    r_{\varepsilon}} |\beta(t) - \beta(t_{\varepsilon})| <
    \frac {\delta_{\varepsilon}}{4}, \ \forall n \ge N_{\varepsilon,2}.
    \end{equation}
    Inequality \eqref{eq:pf4} together with \eqref{eq:pf5} implies
    $$x_{n}(\lambda_{n}(t_{\varepsilon}))\ge
    \beta (\lambda_{n}(t_{\varepsilon})) +
    \frac{\delta_{\varepsilon}}{4}, \ \forall n\ge N_{\varepsilon} :=
    \max\{N_{\varepsilon,1}, N_{\varepsilon,2}\},$$
    or equivalently,
    $$
    \pi(x_{n}) \le \lambda_{n}(t_{\varepsilon})
    \le \lambda_{n}(\pi(x) + \varepsilon), \forall n\ge N_{\varepsilon}.$$
    Finally,
    taking  $\lim\sup_{n\to \infty}$ on each side of the above inequality
    and using $\|\lambda_n - \mathbb I\| \to 0$ of \eqref{eq:pf3}, we have
    $$\lim\sup_{n\to \infty} \pi(x_n) \le
    \lim\sup_{n\to \infty} \lambda_{n}(\pi(x) + \varepsilon)
    =  \pi(x)+ \varepsilon.$$
    So we
    conclude that by the arbitrariness of $\varepsilon$,
    $$\lim\sup_{n\to \infty} \pi(x_n) \le
      \pi(x).$$
   \item Case 2.
    We prove the reverse
    inequality.
    By the definition
    of the first hitting time, for any $\varepsilon>0$,
    there exists $\delta_{\varepsilon}>0$ such that
    \begin{equation}\label{eq:pf6}
     (\beta- x)^*(\pi(x) - \varepsilon) > \delta_{\varepsilon},
     \
     (x- \alpha)^{*}(\pi(x) - \varepsilon) > \delta_{\varepsilon}.
    \end{equation}
    Also, \eqref{eq:pf3} implies that,
    $\|\alpha\circ \lambda_{n}- \alpha\| \to 0$,
    $\|\beta\circ \lambda_{n}- \beta\| \to 0$
    and $ \|x_n \circ \lambda_n - x\| \to 0$. Using the triangle inequality,
    we have
    $$
    \|(x_{n} - \alpha)\circ \lambda_{n} - (x-\alpha)\| \to 0,
    \ \|(x_{n} - \beta)\circ \lambda_{n} - (x-\beta)\| \to 0.$$

    So, there exists a large $N_{\varepsilon}$ such that
    $$
    \|(x_{n} - \alpha)\circ \lambda_{n} - (x-\alpha)\|
    < \frac{\delta_{\varepsilon}}{2},
    \ \|(x_{n} - \beta)\circ  \lambda_{n} - (x-\beta)\|
    < \frac{\delta_{\varepsilon}}{2},
    \ \forall n \ge N_{\varepsilon}.$$
    In particular, this leads to
    $$
    (\alpha - x_{n})\circ \lambda_{n}(t) <
    (\alpha - x)(t) + \frac{\delta_{\varepsilon}}{2}, \
    (x_{n} - \beta)\circ \lambda_{n}(t) <
    (x-\beta)(t) + \frac{\delta_{\varepsilon}}{2},
    \ \forall t \in [0,1], \ n \ge N_{\varepsilon}.
    $$
    By taking $\sup_{[0, \pi(x) - \varepsilon]}$ on both sides of the above
    inequalities and using \eqref{eq:pf6}, we have
    $$
    (x_{n} - \beta)^{*}  \lambda_{n}(\pi(x) - \varepsilon) <
    (x-\beta)^{*}(\pi(x) - \varepsilon) + \frac{\delta_{\varepsilon}}{2}
    < - \frac{\delta_{\varepsilon}}{2}, \ \forall n\ge N_{\varepsilon},
    $$
    and
    $$
    (\alpha - x_{n})^{*}\lambda_{n}(\pi(x) - \varepsilon) <
    (\alpha -x)^{*}(\pi(x) - \varepsilon) + \frac{\delta_{\varepsilon}}{2}
    < - \frac{\delta_{\varepsilon}}{2}, \ \forall n\ge N_{\varepsilon}.
    $$
    Thus, we have
    $$\pi(x_{n}) \ge \lambda_{n}(\pi(x) - \varepsilon),
    \ \forall n\ge N_{\varepsilon}.$$
    Taking $\lim\inf_{n}$ and using \eqref{eq:pf3} and the fact of
    the arbitrariness of $\varepsilon$, we have
    $$\lim\inf_{n}\pi (x_{n}) \ge \pi(x).$$
\end{itemize}

 Summarizing the above, we have proved the continuity of $\pi\cd$
 at $x\in C_{1}$ with respect to the Skorohod topology, i.e.,
 $\lim_{n\to \infty}\pi(x_{n}) = \pi(x)$ whenever $x_{n}\to x\in C_{1}$
 with respect to the Skorohod topology.
 Similar arguments as in Cases 1 and 2 yield that $\pi\cd$ is also
 continuous
 at $x\in C_2$.
\end{proof}

\subsection{Main Convergence Results}
In this section, we utilize the continuity properties under Skorohod topology
together with the continuous mapping theorem to obtain the main convergence
result.
To proceed, we make the following assumptions.

\begin{itemize}
\item [(A1)] If $\beta<\infty$, then $\mathbb{P}\{(X-\beta)^*(1) \ge
  0\}>0$.
  If $\alpha>-\infty$, then
  $\mathbb{P}\{(-X+\alpha)^*(1) \ge 0 \}>0$.
\end{itemize}

Note that Condition
(A1) means that the process hits both barriers
$\alpha\cd$ and $\beta\cd$ with positive probability,
whenever $-\alpha\cd$ or $\beta\cd$ are bounded above.
This is not
a restriction; it is only imposed for the technical convenience since one can
simply
set $\alpha(t) = -\infty$ (resp., $\beta(t) = \infty$) for all $t\in [0,1]$,
if $X\cd$ never hits $\alpha\cd$ (resp., $\beta\cd$) almost surely in $\mathbb P$.
If $- \alpha (t)= \beta(t) = \infty$, then $\tau = 1$.

\begin{itemize}
\item [(A2)]  $X$ satisfies
$\inf\{t>\tau: X(t) \notin [\alpha(t), \beta(t)]\} = \tau$
almost surely  $\mathbb P$.
\end{itemize}

Condition (A2) requires that the
boundary $\partial \Gamma$ is regular with respect to the process
$X\cd$.
Note that for any small $\e>0$, $X$ exits from $\bar \Gamma$ in
the interval $(\tau, \tau+ \e)$ under (A2).
Loosely speaking, the condition (A2) means that
the process $X(t)$ exits $\bar \Gamma$ immediately after it hits the
boundary at $\tau$.
Note that (A2) also implies that
    $$\mathbb{P}\{X\in  \cup_{i=1}^3 C_i\} = 1 \hbox{ or equivalently }
    \mathbb{P}\{X\in  C_{4}\} = 0.$$
 More discussions are referred to Remark~\ref{rem:1}.
\begin{itemize}
\item [(A3)]
 $g: \mathbb{R}^{4m+1} \to \mathbb{R}$ is an
  almost surely
  continuous function
  with respect to
  $\mathbb{P}Z^{-1}$, where
  $Z :=  \Pi(X,  \tau  \nu^{(1)} ), \Pi(X, \nu^{(2)}),
  \Pi(X^{*}, \tau \nu^{(3)}), \Pi(X^{*}, \nu^{(4)}), \tau)$.
  In fact, (A3) is equivalent to the following statement: $g$ is only
  discontinuous at points in a set $\mathcal N$ satisfying
  $\mathbb P\{Z\in \mathcal N\} = 0$;
  see also Remark~\ref{rem:2}.

  \item[(A4)] One of the following  conditions holds:
  \begin{enumerate}
  \item $g$ is a bounded function;
  \item $g$ is a function with linear growth and
  $\{X^{h}(t): h>0, t\in [0,1]\}$ is uniformly
  integrable.
  \end{enumerate}
\end{itemize}

Now, we are ready to answer question (Q1).

\begin{theorem}\label{thm1}
  Assume {\rm(A1)-(A4)}.
  Let $X\cd$ be an ${\cal F}_t$-adapted
  continuous process with initial $X({0}) =x$,
  and $X^{h}\cd$ be a sequence of RCLL processes satisfying
  $X^h \cd\Rightarrow X\cd$ as $h\to 0$.
Then, $\lim_{h\to 0} V^h = V$.
\end{theorem}

\begin{proof}
We rewrite $V$ and $V^{h}$ by
$$V = \mathbb E [G(X)] \hbox{ and } V^{h} = \mathbb E[G(X^{h})],$$
where $G: D[0,1] \mapsto \mathbb R$ is defined by
$$G(x) = g ( \Pi(x,  \pi(x)  \nu^{(1)} ), \Pi(x, \nu^{(2)}),
  \Pi( \mathcal M(x), \pi(x) \nu^{(3)}), \Pi(\mathcal M(x), \nu^{(4)}), \pi(x)).$$

Note that (A2) implies that
$\mathbb{P}\{X\in  \cup_{i=1}^3 C_i\} = 1.$
Together with (A3) and Lemma~\ref{lem:w1}, \ref{l:Pi}, \ref{l:pi},  we have
continuity of $G$ almost surely $\mathbb PX^{-1}$. By the continuous
     mapping theorem
     \cite[Theorem 2.7]{Bil99}, we conclude that
     $$G(X^{h}) \Rightarrow G(X) \hbox{ as } h\to 0.$$
Together with (A4), it results in $\lim_{h\to 0} V^h = V$;
    see \cite[P. 25, 31]{Bil99}.
\end{proof}

Theorem~\ref{thm1} holds under assumptions of (A1)-(A4).
Recall that  (A1) is not
a restriction.
We elaborate on (A2),
(A3), and (A4)
in what follows.

\begin{remark} [Discussions on (A2) and Example~\ref{ex:t}]
\label{rem:1}
{\rm In fact, (A2) is a requirement on the regularity of the boundary
$\partial \Gamma$ with respect to the process $X$,
and it is referred to as $\tau'$-regularity for simplicity; see \cite{SV72}.
Note that
since $X\cd$ in Example~\ref{ex:t} violates $\tau'$-regularity (A2), by observing
$$\inf\{t>\tau: X(t) \notin [\alpha(t), \beta(t)]\} \wedge 1 = 1
> 1/2 = \tau,$$
it yields the convergence to the wrong value. In other words,
(A2) is crucial for the investigation of the convergence.
}\end{remark}

\begin{remark}
 [Discussions on (A3) and Example~\ref{e:barrier}]
 \label{rem:2}
{\rm Assumption (A3) is the requirement on the function $g$.
First of all, it allows discontinuity of $g$, but it cannot be too much
discontinuous in the sense that
it is at least required to be almost surely continuous.
However, it is already enough to  include option pricing for the
discontinuous payoff such as the barrier option in Example~\ref{e:barrier}.
In fact, $g$ of Example~\ref{e:barrier} given by
 $$g (x_{1}, x_{2}, \ldots, x_{4m+1}) = e^{-r} (x_{2m} - \frac 1 2)^{+} I_{[1,\infty)}(x_{4m}),$$ is
 continuous only at $\{x_{4m} = 1\}$. Suppose the stock price $X$ follows a geometric Brownian motion, then the probability measure $\mathbb P$ satisfies $\mathbb P\{X^{*}(1) = 1\} = 0$, and $g$ is continuous almost surely $PZ^{-1}$.
}
\end{remark}

\begin{remark}
 [Discussions on (A4) and Example~\ref{ex:c2}]
 {\rm
In (A4), another issue yet
mentioned is the growth condition of $g\cd$.
In particular,
if $g\cd$ is of linearly growth function of the underlying price like in the call type option, then one can verify the uniform
integrability. We have already seen that, for instance, Example~\ref{ex:c2} converges to a wrong
value, since $\{X(t): t\in (0,1)\}$ is not uniformly integrable while
the payoff function $g(x) = x$ is linear growth.
Recall the definition of
uniform integrability.
A set of random variables $\{Y_{\gamma}: \gamma\in \Gamma\}$
is said uniformly integrable, if for any $\e>0$, there exists
a compact set $K$ such that
$$\sup_{\gamma\in \Gamma}\mathbb E[Y_{\gamma} I_{\{Y_{\gamma}\notin K\}}]<\e.$$
Note that to ease the verification of the
uniform integrability, one can often use Proposition~\ref{p:ui}  practically rather than the above definition.

}\end{remark}

\section{Weak Convergence of Markov Chain Approximation}
\label{sec:mar}

\subsection{Markov Chain Approximation}

In this section, we
establish the weak convergence of Markov chain approximation for multidimensional stochastic differential equations. In what follows,
$K$ is a generic constant whose value may change at each line.
\begin{enumerate}
 \item [(A5)] $b$ and $\sigma$
 are
 Lipschitz
 in $y$ and
 H\"older-1/2
 continuous in $t$, i.e., with $\phi = b$, $\sigma$,
 $$|\phi(y_{1},t_{1}) - \phi(y_{2},t_{2})|\le  K (|y_{1} - y_{2}| + |t_{1}-t_{2}|^{1/2}).$$
\end{enumerate}

Let $Y = \{Y(t): t\in [0,1]\}$ be the unique solution of
\begin{equation}\label{eq:Y}
 d Y(t) = b(Y(t), t) dt + \sigma(Y(t), t) dW(t); \ Y({0}) = y,
\end{equation}
where $b: \mathbb R^{d+1} \mapsto \mathbb R^{d}$, $W$ is
a standard $\mathbb R^{d_{1}}$ Brownian motion, and
$\sigma: \mathbb R^{d+1} \mapsto \mathbb R^{d\times d_{1}}$.
Let $t_{0}^{h}  = 0 \le t_{1}^{h}\le \cdots  \le t_{N}^{h} = 1$ be a
sequence of increasing predictable (i.e., $t_{i}^{h}$ is $\mathcal{F}_{i-1}^{h}$-measurable.) random times with respect to
a discrete filtration $\{\mathcal F_{i}^{h}: i = 0, 1, \ldots\}$,
and $\{Y_{i}^{h}: i = 1, 2, \ldots N\}$
be a sequence of  $\{\mathcal F_{i}^{h}\}$-adapted Markov chain  in
$\mathbb R^{d}$ with transition probability
$$\mathbb P\{Y_{i+1}^{h} \in d y | Y_{i}^{h} = x, t_{i}^{h} = t\} = p^{h}(t,x,y).$$
We use $Y^{h} = \{Y^{h}(t): t\in [0,1]\}$ to denote piecewise constant interpolation
\begin{equation}\label{eq:Yh}
 Y^{h}(t) = \sum_{i = 0}^{n-1} Y^{h}_{i} I_{\{t_{i}^{h}\le t < t_{i+1}^{h}\}}.
\end{equation}
For notational simplicity, we set
$\Delta t_{n}^{h} = t_{n+1}^{h} -t_{n}^{h}$ and $\Delta Y_{n}^{h} = Y_{n+1}^{h} - Y_{n}^{h}$, $\mathbb E_{n}^{h}[ \ \cdot \ ] := \mathbb E[\ \cdot \  | \mathcal{F}_{n}^{h}]$,
 and
  $$z^{h}(t) := \max\{ n\ge 0: t_{n}^{h} \le t\}, \quad
  \Delta M_{n}^{h} = \Delta Y_{n}^{h} - \mathbb E_{n}^{h} [\Delta Y_{n}^{h}].$$
The interpolation of the
 Markov chain process
$Y^{h}$ is said to be {\it locally consistent}, if
\begin{enumerate}
\item [(LC1)] $\mathbb E_{n}^{h} [\Delta Y_{n}^{h}] =
   \mathbb E_{n}^{h} [\Delta t_{n}^{h}] \cdot (b(Y_{n}^{h}, t_{n}^{h}) + O(h) )$,
\item [(LC2)] $\hbox{cov}  (\Delta Y_{n}^{h} | \mathcal{F}_{n}^{h}) =
 \mathbb E_{n}^{h} [\Delta t_{n}^{h}] \cdot
  ((\sigma \sigma')(Y_{n}^{h}, t_{n}^{h}) + O({h}))$,
where $O(h)$ is either
a $d$-dimensional vector
or $d\times d$ dimensional matrix that is ${\mathcal F}^h_n$-measurable with each element being $O(h)$.
\end{enumerate}

To proceed, we also require quasi-uniform step size.
\begin{enumerate}
 \item [(QU)]  The step size $\{\Delta t_{i}^{h}\}$ satisfies $ \frac h K \le \inf_{i} \Delta t_{i}^{h} \le \sup_{n}\Delta t_{n}^{h} \le Kh$.
\end{enumerate}
The (QU) condition yields
\begin{equation}\label{eq:qu}
 z^{h} (t) \le Kt /h, \quad z^{h}(t) \sup_{i}\Delta t_{i}^{h} \le K t \
 \hbox{ almost surely. }
\end{equation}

The main goal of this section is to show that $\{Y^{h}(t): t\in [0,1], h>0\}$ is uniformly integrable and  $Y^{h}\cd \Rightarrow Y\cd$ as $h\to 0$. Since
the entire proof is rather long, we first provide some useful estimates.

\begin{lemma}
 \label{l:est1}
 For the stochastic differential equation \eqref{eq:Y},
assume that  {\rm(A5)} is satisfied, and $Y^{h}\cd$ of \eqref{eq:Yh} is a sequence of random  processes
 satisfying the local consistency
 conditions $($LC1-LC2$)$ and quasi-uniform step size $($QU$)$. Then, the family of random variables
 $\{Y^{h}(t): t\in [0,1], h>0\}$ satisfies:
 \begin{equation}\label{eq:est2}
\mathbb E [|Y^{h}(t)|^{2}] \le K
\hbox{ for all } t\in [0,1],
\end{equation}
and
\begin{equation}
 \label{eq:est3}
 \mathbb E \Big[\sup_{0\le n \le z^{h}(t) - 1} |\sum_{i = 0}^{n} \Delta Y_{i}^{h}|^{2} \Big]
\le Kt .
\end{equation}

\end{lemma}

\begin{proof}
 First, we separate the entire estimation into two parts.
 $$
\begin{array}{ll}
 \mathbb E[\sup_{0\le n \le z^{h}(t) - 1} |\sum_{i = 0}^{n} \Delta Y_{i}^{h}|^{2}]
  = \mathbb E [\sup_{0\le n \le z^{h}(t) - 1}
 |\sum_{i = 0}^{n} \mathbb E_{i}^{h}[\Delta Y_{i}^{h}] + \Delta M_{i}^{h}|^{2}]
 \\ \quad \quad  \le \displaystyle  K
 \mathbb E \Big[\sup_{0\le n \le z^{h}(t) - 1}
 \Big|\sum_{i = 0}^{n} \mathbb E_{i}^{h}[\Delta Y_{i}^{h}] \Big|^{2} \Big]
 + K
 \mathbb E \Big[\sup_{0\le n \le z^{h}(t) - 1}
 \Big|\sum_{i = 0}^{n} \Delta M_{i}^{h}\Big|^{2} \Big] := K \wdt I_1 + K \wdt I_2.
\end{array}
 $$
 Note that $\wdt I_{1}$ has the following upper bound by local consistency,
  $$\wdt I_{1} \le
 \mathbb E \Big[ z^{h}(t)
 \sum_{i = 0}^{z^{h}(t) - 1} \Big|\mathbb E_{i}^{h}[\Delta Y_{i}^{h}] \Big|^{2} \Big]
 \le K \mathbb E \Big[ z^{h}(t)
 \sum_{n = 0}^{z^{h}(t) - 1}
  (\mathbb E_{n}^{h} [\Delta t_{n}^{h}])^{2}
  \cdot (b^{2}(Y_{n}^{h}, t_{n}^{h}) + O(h^{2}) )\Big].$$
Owing to (QU) condition, we can use the inequality \eqref{eq:qu} to obtain
$$
\begin{array}
 {ll}
 \wdt I_{1}
 & \le Kt \mathbb E \Big[
 \sum_{n = 0}^{z^{h}(t) - 1}
  \mathbb E_{n}^{h} [\Delta t_{n}^{h}]
  \cdot (b^{2}(Y_{n}^{h}, t_{n}^{h}) + O(h^{2}) )\Big]\\
  & \le \disp Kt \mathbb E \Big [
 \sum_{n = 0}^{[Kt/h] -1}
  \mathbb E_{n}^{h} [\Delta t_{n}^{h}
  \cdot (b^{2}(Y_{n}^{h}, t_{n}^{h}) + O(h^{2})) \cdot I_{\{n\le z^{h}(t) - 1\}}]
  \Big].
\end{array}
$$
In the last line above, the term $(b^{2}(Y_{n}^{h}, t_{n}^{h}) + O(h^{2})) \cdot I_{\{n\le z^{h}(t) - 1\}}$ can be included into $\mathbb E_{n}^{h}$,
since it is $\mathcal{F}_{n}^{h}$-measurable. Now we are ready to use
the tower property of conditional expectation and regularity condition (A5) and end up with
\begin{equation}
 \label{eq:i1}
 \wdt I_{1} \le Kt \mathbb E \Big[\int_{0}^{t} |Y_{s}|^{2} ds \Big] +  Kt^{2} h^{2}.
\end{equation}

On the other hand,
$$
\begin{array}{ll}
  \wdt I_2 &\disp = \mathbb E \Big[\sup_{0\le n \le z^{h}(t) - 1}
 \Big|\sum_{i = 0}^{n} \Delta M_{i}^{h}\Big|^{2} \Big] \\
 &\disp \le K \mathbb E \Big[
 \Big|\sum_{i = 0}^{ z^{h}(t) - 1} \Delta M_{i}^{h}\Big|^{2} \Big]
 \ \hbox{ by Doob's Maximal inequality }\\
 &\disp = K\mathbb E \Big[
 \sum_{i = 0}^{ z^{h}(t) - 1} |\Delta M_{i}^{h}|^{2} \Big]
 \ \hbox{ since } \ \mathbb E[ (\Delta M_{i}^{h})' (\Delta M_{j}^{h})] = 0,
  \forall i\neq j \\
  &\disp \le K  \mathbb E \Big[
 \sum_{i = 0}^{ z^{h}(t) - 1}
 \Delta t_{i}^{h} \Trace(\sigma \sigma') (Y_{i}^{h}, t_{i}^{h})
  \Big] +K ht \
 \hbox{ by  (LC2)}.
 \end{array}
$$
where $\Trace(A)$ denotes the trace of $A$. Together with (A5), this
implies that
\begin{equation}
 \label{eq:i2}
 \wdt I_2 \le  K \mathbb E \Big[ \int_{0}^{t} |Y_{s}^{h}|^{2}ds
  \Big] + Kht.
\end{equation}

 Combining \eqref{eq:i1} and \eqref{eq:i2}, it yields that
\begin{equation}\label{eq:i12}
 \mathbb E \Big[\sup_{0\le n \le z^{h}(t) - 1} |\sum_{i = 0}^{n} \Delta Y_{i}^{h}|^{2} \Big]
\le Kt h +  K \mathbb E\Big[\int_{0}^{t} |Y^{h}(s)|^{2}ds\Big].
\end{equation}
Therefore, we have
  $$
\begin{array}
 {ll}
 \mathbb E [|Y^{h}(t)|^{2}] & \le K \mathbb E[|y|^{2}+
 |\sum_{i=0}^{z^{h}(t)-1} \Delta Y^{h}_{i}|^{2}]  \le K + K \mathbb E[\int_{0}^{t} |Y^{h}(s)|^{2}ds].
\end{array}
$$
Gronwall's inequality then yields the result of \eqref{eq:est2}.
Plug in \eqref{eq:est2} to \eqref{eq:i12}, we conclude \eqref{eq:est3}.
\end{proof}

To proceed, we also need the following convenient proposition for the
uniform integrability, and the
reader is referred to  \cite{Dur05} for the proof of the proposition.
\begin{proposition}\label{p:ui}
 A family of  real-valued
 random variable $\{Z_{\theta}: \theta\in \Theta\}$
 for some index set $\Theta$
 is uniformly integrable if
 $\sup_{\theta} \mathbb E[ |Y_{\theta}|^{p}] <\infty$
 for some $p>1$.
\end{proposition}

\begin{theorem}\label{t:wc}
Under the same assumptions as that of Lemma~\ref{l:est1}, the family of random variables
 $\{Y^{h}(t): t\in [0,1], h>0\}$ is uniformly integrable, and  $Y^{h}\cd \Rightarrow Y\cd$ as $h\to 0$.

\end{theorem}

\begin{proof}
we conclude first
$\{Y^{h}(t): 0\le t \le T, h>0\}$ is uniformly integrable
by Proposition~\ref{p:ui} together with \eqref{eq:est2}.
We
divide the rest of the proof into several steps.
\begin{enumerate}
 \item
  We
  consider the tightness of
  $\{\mathbb P_{h} = \mathbb P (Y^{h})^{-1}: h>0\}$. It is enough to
  verify conditions imposed on Theorem 13.3.2 and the subsequent
  corollary given in \cite{Bil99}.
\begin{enumerate}
   \item By Chebyshev's inequality and \eqref{eq:est2}
   $$
\begin{array}
 {ll}
\disp \lim_{a\to \infty} \lim\sup_{h\to 0} \mathbb P\{ |Y^{h}(t)| \ge a \}
 & \disp\le
 \lim_{a\to \infty} \lim\sup_{h\to 0} \frac 1 {a^{2}} \mathbb E[ |Y^{h}(t)|^{2} ]
  = 0.
\end{array}$$
\item
For the purpose of characterization of tightness
for
discontinuous
functions, we need to introduce some notions of modulus of continuity
$\omega(Y^{h}, \delta)$ and $\omega'(Y^{h}, \delta)$ as the following.
First, define
$\omega(Y^{h}, \mathcal I) =
\sup_{r_{1},r_{2}\in \mathcal I} |Y^{h}(r_{1}) - Y^{h}(r_{2})|$ for any subset
$\mathcal I \subset [0,1]$. We also use $\mathcal T(\delta)$ to denote
the collection of all $\delta$-sparse partitions of $[0,1]$.
 Then, for any $\delta>0$, we can define
$$ \omega(Y^{h}, \delta) =
\sup_{0\le t\le 1- \delta} \omega(Y^{h}, [t,t+\delta]), \
 $$
 and
 $$\omega'(Y^{h}, \delta) =
 \inf_{\{t_{i}\} \in \mathcal T(\delta)} \max_{i} \omega(Y^{h}, [t_{i-1}, t_{i}) ).
 $$

For the purpose of tightness of discontinuous
functions, we also need to introduce modified version of modulus of
continuity.
Thanks to  \eqref{eq:est3}, we have, for an arbitrary $\varepsilon >0$,
$$
\begin{array}
 {ll}
 \displaystyle
 \lim_{\delta \to 0} \lim\sup_{h\to 0} \mathbb P\{\omega'(Y^{h}, \delta) \ge
 \varepsilon \}\\ \displaystyle
 \quad\disp \le
 \lim_{\delta \to 0} \lim\sup_{h\to 0} \mathbb P\{\omega(Y^{h}, 2\delta) \ge
 \varepsilon \} \\ \displaystyle
 \quad\disp \le
 \lim_{\delta \to 0} \lim\sup_{h\to 0}\frac{1}{\varepsilon^{2}}
 \mathbb E [\omega^{2}(Y^{h}, 2\delta)]\\
 \quad\disp \le \displaystyle
 \lim_{\delta \to 0} \lim\sup_{h\to 0}\frac{1}{\varepsilon^{2}}
 \mathbb E
 \Big[\sup_{z^{h}(t) - 1 \le n \le z^{h}(t+2\delta) - 1} |\sum_{i = 0}^{n} \Delta Y_{i}^{h}|^{2} \Big] \\
 \quad \le  \displaystyle
 \lim_{\delta \to 0} \lim\sup_{h\to 0}\frac{1}{\varepsilon^{2}}
  K \delta = 0.
\end{array}
$$\end{enumerate}
As a result, $\{\mathbb P_{h} = \mathbb P (Y^{h})^{-1}: h>0\}$ is tight.

  \item
  Since $Y^{h}\cd$ is tight, for an arbitrary infinite sequence,
 there exists a subsequence
 that has a weak limit. For
 notational convenience, we denote this
 subsequence again by $\{Y^{h}\cd\}$, and its limit by $\bar Y$.
 Due to uniqueness of weak solution, it is enough to show that
  $\bar Y\cd$ is the weak solution
  of \eqref{eq:Y}.
 Since $Y^h\cd\Rightarrow \bar Y\cd$,
the uniform integrability and (A5) lead to
 $$\begin{array}{ll}
 &\disp \lim_{h\to 0} \mathbb E[Y^{h}(t)] = \mathbb E[\bar Y(t)], \ \hbox{ and }\\
 & \disp \lim_{h\to 0} \mathbb E[b(Y^{h}(t), t)] = \mathbb E[b(\bar Y(t),t)].\end{array}$$
  Therefore, if we set
  $$M(t) := \bar Y(t) - \bar Y(0) - \int_{0}^{t} b(\bar Y(s), s) ds,$$
  we have
\begin{equation}\label{eq:pf7}
\begin{array}{ll}
 \mathbb E[M(t)] &=  \lim_{h\to 0} \mathbb E
 [Y^{h}(t) - Y^{h}(0) - \int_{0}^{t} b(Y^{h}(s), s) ds]
 \\ & = \displaystyle \lim_{h\to 0} \mathbb E \Big[\sum_{i=1}^{z^{h}(t)}
 ((Y_{i}^{h} - Y_{i-1}^{h}) - \int_{t_{i-1}^{h}}^{t_{i}^{h}} b(Y_{i-1}^{h},t) dt)\Big]
 - \lim_{h\to 0} \mathbb E \Big[\int_{t_{z^{h}(t)}^{h}}^{t} b(Y_{z^{h}(t)}^{h},s) ds\Big]
 \end{array}
\end{equation}
 Regarding the last term  above, by linear growth of $b$
 $$
 \lim_{h\to 0} \mathbb E
 \Big[ |\int_{t_{z^{h}(t)}^{h}}^{t}  b(Y_{z^{h}(t)}^{h},s) ds|
 \Big] \le \lim_{h\to 0}
 \mathbb E \Big[\int_{t_{z^{h}(t)}^{h}}^{t} K + K |Y_{z^{h}(t)}^{h}| ds \Big]
 \le K  \lim_{h\to 0}
 \mathbb E [ \Delta t_{z^{h}(t)}^{h} (1 + |Y^{h}(t)|)
 ] .
 $$
  Therefore, by (QU) and \eqref{eq:est2}, we conclude that the second term on \eqref{eq:pf7} satisfies
 $$\lim_{h\to 0} \mathbb E \Big[\int_{t_{z^{h}(t)}^{h}}^{t} b(Y_{z^{h}(t)}^{h},s) ds\Big] = 0.$$
As for the first term of \eqref{eq:pf7}, using the H\"older continuity in $t$
and the tower property on local consistency
$$
\begin{array}{ll}
 \disp\lim_{h\to 0} \mathbb E \Big[\sum_{i=1}^{z^{h}(t)}
 ((Y_{i}^{h} - Y_{i-1}^{h}) - \int_{t_{i-1}^{h}}^{t_{i}^{h}} b(Y_{i-1}^{h},t) dt)\Big]
 \\ \quad\disp = \lim_{h\to 0} \mathbb E \Big[\sum_{i=1}^{z^{h}(t)}
 ((Y_{i}^{h} - Y_{i-1}^{h}) -
 b(Y_{i-1}^{h},t) \Delta t_{i-1}^{h} + O(|\Delta t_{i-1}^{h}|^{3/2}))\Big]
 \\ \quad\disp = \lim_{h\to 0} \mathbb E \Big[\sum_{i=1}^{z^{h}(t)}
 ( \Delta t_{i-1}^{h} O(h) + O(|\Delta t_{i-1}^{h}|^{3/2}))\Big] = 0.
\end{array}
$$
Therefore, $\mathbb E [M(t)] = 0$. In fact, one uses exactly the same procedure to show $\mathbb E[M(t)|\mathcal F_{s}] = M(s)$ for any
$0\le s \le t$, and it concludes $M(t)$ is a martingale.

Next, we use $\bar Y_{l}\cd$ and $Y^{h}_{l}\cd$ to denote the $l$th component of
the vector process $\bar Y\cd$ and $Y^{h}\cd$, and use
$\langle \bar Y_{l}, \bar Y_{m}\rangle(t)$ to denote the cross-variation of
two real processes $Y_{l}\cd$ and $Y_{m}\cd$ up to time $t$. Then,
$$
\begin{array}{ll}
\disp \mathbb E |\langle \bar Y_{l}, \bar Y_{m}\rangle(t) - \int_{0}^{t}
 (\sigma_{l}\sigma_{m})(\bar Y(s),s) ds|
 \\ \quad\disp = \lim_{h\to 0}
 \mathbb E |\langle Y^{h}_{l}, Y^{h}_{m}\rangle(t) - \int_{0}^{t}
 (\sigma_{l}\sigma_{m})(Y^{h}(s),s) ds|
 \\ \quad\disp =
 \lim_{h\to 0} \mathbb E  |\sum_{i=1}^{n}
 (Y_{l,i}^{h} - Y_{l,i-1}^{h})  (Y_{m,i}^{h} - Y_{m,i-1}^{h})
 - \int_{0}^{t}
 (\sigma_{l}\sigma_{m})(Y^{h}(s),s) ds| = 0.
\end{array}
$$
The last equality can be
obtained similar to the above owing to
$\mathbb E[M(t)] = 0$,
using the
local consistency, and
regularity assumption (A5) on $\sigma$. Therefore,
the cross-variation $\langle \bar Y_{l}, \bar Y_{m}\rangle(t) = \int_{0}^{t}
 (\sigma_{l}\sigma_{m})(\bar Y(s),s) ds $.
 This again implies that $M(t)$ is a $d$-dimensional martingale process with
 its quadratic variation $\langle M\rangle(t) = \int_{0}^{t} (\sigma \sigma')(\bar Y(s),s) ds$. Applying Levy's martingale characterization
 on the time-changed Brownian motion, there exists a $d_{1}$-dimensional Brownian motion $B(t)$ such that
 $M(t) = \int_{0}^{t}\sigma(\bar Y(s), s) dB(s)$.
Therefore, $\bar Y$ is the weak solution of \eqref{eq:Y}.
\end{enumerate}
Summarizing all the above, it leads to the conclusion.
\end{proof}

\subsection{Examples of Weak Convergence of  Markov Chain Approximation}

The above construction of Markov chain approximation is based on the
local consistency, which is slightly different from the local consistency given
by   \cite[Theorem 10.4.1]{KD01}. As a result, the convergence result of
the Markov chain approximation is generalized in the following sense.
 $\sigma$ and $b$ may be unbounded but have linear
growth. Therefore, the geometric Brownian motion is covered by weak convergence result of Theorem~\ref{t:wc} as an important application.
In fact, locally consistent MC approximation is flexible for its
various choices. For illustrations, we
give several  simple MC approximations for one dimensional
process, which is not included in \cite{KD01}.

\begin{example}
 \label{ex:euler}
{\rm Euler approximation can be considered as a special case
  of MC approximations.
 Let $\{Y_n^h\}$ be a Markov chain generated by
\begin{enumerate}
\item $Y_0^h = y, t_0^h = 0$.
\item Let the
  transition probability be
  $\Delta t_{n}^{h} = h$
  \begin{equation}
  \label{eq:phe}
  Y^h_{n+1} = Y_n^h + b(Y_n^h,nh) h
  + \sigma(Y_n^h,nh) \sqrt h N_{n} ,
\end{equation}
where $\{N_{n}\}$ is a sequence of i.i.d. standard normal random variables.
One can easily verify  all local consistency conditions.
Then  assuming (A5), Theorem~\ref{t:wc}
implies that the piecewise constant interpolation $Y^h$ defined in \eqref{eq:Yh}  converges weakly to $Y$,
the solution  to \eqref{eq:Y}.
\end{enumerate}
}\end{example}

\begin{example}\label{ex:mc1}
{\rm The following MC approximation can be considered as
an extension of binomial approximation of Brownian motion.
Let $d = d_{1} = 1$.
Let $\{Y_n^h\}$ be a Markov chain generated by
\begin{enumerate}
\item $Y_0^h = y, t_0^h = 0$.
\item Let the
  transition probability be, with $\Delta t_{n}^{h} = h$
  \begin{equation}
  \label{eq:ph1}
  \mathbb{P}\Big(Y^h_{n+1} = Y_n^h + b(nh, Y_n^h) h
  \pm \sigma(nh, Y_n^h) \sqrt h \Big| Y_n^h\Big) = 1/2
\end{equation}
\end{enumerate}

The above Markov chain is locally consistent since
direct computation leads to
\begin{enumerate}
\item $\mathbb E [\Delta Y_{n}^{h}| Y_{n}^{h} = y, t_{n}^{h} = t] =
   \mathbb E [\Delta t_{n}^{h}| Y_{n}^{h} = y, t_{n}^{h} = t] \cdot b(y, t)$,
\item ${\rm cov}  (\Delta Y_{n}^{h} | Y_{n}^{h} = y, t_{n}^{h} = t) =
 \mathbb E [\Delta t_{n}^{h}| Y_{n}^{h} = y, t_{n}^{h} = t] \cdot
 \sigma^{2} (y, t)$.
\end{enumerate}
Therefore,  assuming (A5), the piecewise constant interpolation $Y^h\cd$
of the form \eqref{eq:Yh}  is convergent to $Y\cd$ of \eqref{eq:Y} by Theorem~\ref{t:wc}.

}\end{example}

\begin{example}\label{ex:mc2}
{\rm Assume $|\sigma|\wedge |1/\sigma| >\epsilon>0$ in addition to (A5). Then one can use  a binomial tree type approximation
of the diffusion term
$\sigma(Y(t), t) dW(t)$ in \eqref{eq:Y}.
Let $d = d_{1} =1$.
Let $\{Y_n^h\cd\}$ be a Markov chain generated by
\begin{enumerate}
\item $Y_0^h = y, t_0^h = 0$.
\item Let $ \Delta t_n^h (Y_n^h, t_n^h) =
   \frac{h}{\sigma^2 (Y_n^h, t_n^h)}$, and the
  transition probability be
  \begin{equation}
  \label{eq:ph2}
\begin{array}
 {l}
  \mathbb{P}\Big(Y^h_{n+1} = Y_n^h + b(Y_n^h, t_n^h)
  \Delta t_n^h (Y_n^h, t_n^h)
  \pm \sqrt h \Big| (Y_n^h, t_{n}^{h})\Big) = 1/2
.
\end{array}
 \end{equation}
\end{enumerate}
Note that this Markov chain satisfies (QU) as well as  local consistency
since
\begin{enumerate}
\item $\mathbb E [\Delta Y_{n}^{h}| Y_{n}^{h} = y, t_{n}^{h} = t] =
   \mathbb E [\Delta t_{n}^{h}| Y_{n}^{h} = y, t_{n}^{h} = t] \cdot b(y, t)$,
\item ${\rm cov}  (\Delta Y_{n}^{h} | Y_{n}^{h} = y, t_{n}^{h} = t) =
 \mathbb E [\Delta t_{n}^{h}| Y_{n}^{h} = y, t_{n}^{h} = t] \cdot
 (\sigma^{2} (y, t)+O(h))$.
\end{enumerate}
Therefore,  the piecewise constant
interpolation $Y^h\cd$ of the form \eqref{eq:Yh}  is convergent to $Y\cd$ of \eqref{eq:Y} by Theorem~\ref{t:wc}.
}\end{example}

\subsection{Can We Expect the Strong Convergence?}
Can We expect strong convergence?
In general, the answer is no.  To illustrate,
we consider
a special case of Euler approximation of Example~\ref{ex:euler}.
Let $(\Omega, \mathcal{F}, \mathbb{P})$ be a probability
space, on which
$\mathcal{F}_t$ is filtration satisfying usual conditions, and
$W_t$ of \eqref{eq:Y} is
a standard 1-d Brownian motion.
We construct strong approximation of Euler-Maruyama's method
by taking  on $N_{n}$ of \eqref{eq:phe} by
$$
N_{n} = \frac{W_{n+1} - W_{n}}{\sqrt h}.
$$

Under assumption (A5), the  SDE \eqref{eq:Y} has a
unique strong solution.
Suppose
$\hat Y^{h}\cd$ is a continuous interpolation of
Euler approximation of $Y^{h}\cd$ given by
$$\hat Y^h(t) = \hat Y_{nh}^{h} + b(nh, \hat Y_{nh}^{h})
(t- nh) +
\sigma(nh, \hat Y_{nh}^{h}) (W(t) - W(nh)), \quad
\hbox{ for } t\in [nh, nh + h).$$
Then a classical result (see e.g.  \cite[Theorem 2.7.3]{Mao07}) shows that
$$\mathbb{E} \Big[
\sup_{0\le t\le T} |Y(t) - \hat Y^h(t)| \Big] \le
K h^{1/2}.$$

However, the above inequality fails for the piecewise
constant interpolation of EM approximation $\{Y^{h}\}$.
Otherwise, we have following simple counter example.
Consider EM approximation of $W_t$ on $[0,1]$ by
equal step size $h = 1/N$.
Then, we have
$$\mathbb{E}\Big[ \sup_{0\le t \le 1} |W(t) - W([Nt]/N)| \Big]=
\mathbb{E} \Big[ \sup_{1\le n \le N}
\sup_{(n-1)/N \le t < n/N} |W(t) - W(\frac{n-1}{N})| \Big].
$$
Note that, $\bar W(t) = \sqrt N W(t/N)$ is a standard
Brownian motion w.r.t. a time-scaled filtration. So one can reduce the
above equality as
$$\mathbb{E}\Big[ \sup_{0\le t \le 1} |W(t) - W([Nt]/N)| \Big]=
\frac 1 {\sqrt N} \mathbb{E} \Big[
\sup_{1\le n \le N} \Lambda_n \Big],
$$
where  $\{\Lambda_n\}$ are i.i.d. random variables defined by
$$\Lambda_n = \sup_{n-1 \le t < n} | \bar W(t) - \bar W(n-1) |.$$
Since, $\Lambda_n$'s are unbounded iid random variables,
$\mathbb{E} \Big[ \sup_{1\le n \le N} \Lambda_n \Big]$ goes to infinity
as $N\to \infty$. This shows that
$$\mathbb{E}\Big[ \sup_{0\le t \le 1} |W(t) - W([Nt]/N)| \Big] > O(N^{-1/2}).$$
In conclusion, one cannot expect more than weak convergence
merely under local consistency.

\section{Ramification and Further Remarks}\label{sec:fur}
\subsection{Application to Discrete-Monitoring-Barrier Option Underlying Stochastic Volatility} 
\label{sec:vol}
We begin this section with an
application of Theorem~\ref{thm1} to
the following stochastic volatility model; see \cite{BKX12}.
Let $W$ and $B$ be two standard Brownian motions
with correlation $\rho$ in the filtered probability space
$(\Omega, \mathcal{F}, \mathbb{P}, \mathcal{F}_t)$ . Suppose
that the stock price follows
\begin{equation}
 \label{eq:stk}
  d X(t) = X(t) (r dt +  \sigma(Y(t)) d W(t))
\end{equation}
with initial
$X(0) = x > 0$,
and  that the volatility follows
\begin{equation}
 \label{eq:vol}
 d Y(t) = Y(t) (\mu(t)  d t + b(t) d B(t))
\end{equation}
with initial $Y(0) = y >0$.
We consider a discrete-monitoring-barrier option price
$$
 V = e^{-r}\mathbb{E}\Big[ \Big(X(1) - \frac 1 2\Big)^{+}
 I_{[1,\infty)}(\max_{1\le i \le m} X(i/m)) \Big].
 $$
  given by
\eqref{eq:db1} of Example~\ref{e:dbarrier}  by the discrete scheme \eqref{eq:Vh}.
$$
  V^{h} =  \mathbb E [g ( \Pi(X^{h},  \tau^{h}  \nu^{(1)} ), \Pi(X^{h}, \nu^{(2)}),
  \Pi(X^{h,*}, \tau^{h} \nu^{(3)}), \Pi(X^{h,*}, \nu^{(4)}), \tau^{h})],
$$
where $g:\mathbb R^{4m+1} \mapsto \mathbb R$ is of the form
$$
g (x_{1}, x_{2}, \ldots, x_{4m+1}) = e^{-r} (x_{2m} - \frac 1 2)^{+} I_{[1,\infty)}(\max_{m+1\le i \le 2m} x_{i}).$$

One can check that $X\cd$ and $Y\cd$ are unique nonnegative
strong solution of SDEs \eqref{eq:stk} and \eqref{eq:vol}, if
$\sigma$ satisfies polynomial growth, and $\mu$ and $b$ are
 H\"older-1/2 continuous.
In addition, we assume
\begin{center}
non-degeneracy of $X$, i.e.,
$\sigma(y)>0$ for all $y>0$.
\end{center}
Note that it is possible to have
$\sigma(0) = 0$ under the above assumption. In this below,
we examine the convergence
$\lim_{h} V^{h} = V$ when
$X^{h}$ is constructed by the Euler approximation of Example~\ref{ex:euler}.

Note that the regularity condition (A2) is satisfied  by
\cite[Proposition A.1]{SYZ12} because  $\sigma(\cdot)>0$.
On the other hand,
the payoff function $g$
is only discontinuous at the points in the set
$\{\max_{m+1\le i \le 2m} x_{i} = 1\}$. Thanks to the fact
$$\mathbb P\{\max_{1\le i \le m} X(i/m) = 1\} = 0,$$
it implies that $g$ is almost surely
continuous with respect to $\mathbb P$.
Another thing yet to be verified is
the uniform integrability of $\{X^{h}(t): h, t\}$,
since $g$ is linear growth.
Thanks to Theorem~\ref{t:wc} together with (A5), we have
$X^{h}\cd \Rightarrow X\cd$ and the desired uniform integrability
holds.
Therefore, we reach the affirmative answer $\lim_{h} V^{h} = V$.

For the simple demonstration,
we present a computational result on the above example
with the following data. Let the initial stock price be $X_{0} = 0.8$. For
simplicity, we assume constant interest rate and volatility
$r = .1$ and $\sigma = .3$. Suppose stock is monitored monthly, i.e. $m = 12$. If we compute $k = 5000$ many sample paths, and each sample
path is computed by Euler method with $n = 60000$ subintervals, then
the computational result shows that the $95\%$ interval is $[ 0.2310, 0.2364]$. The total Matlab running time on Macbook Air is 194 seconds.
The Matlab code is also available for download at http://01law.wordpress.com/2013/07/18/code/ .

\subsection{Further Remarks}
This work has been devoted to analyzing approximation to path-dependent
functionals.
Using the methods of  weak convergence, we have
provided a unified approach for proving
 the convergence of numerical approximation of path-dependent
 functionals for a wide range of applications.

 A possible
 alternative approach to
 study the convergence may be to evaluate the approximating problem by perturbing the boundary
 $(\alpha(t), \beta(t))$ to $(\alpha(t) - h, \beta(t) +h)$. For instance, \cite{SY09} has studied the property of the value function using the aforementioned 
 perturbation when the value function is non-path dependent and HJB equation is available. It is interesting to check if a similar approach
 works for the path-dependent case.

This paper focused on diffusion models. For future work, it is worthwhile to examine
systems driven by pure jump processes, jump diffusions, and systems with an additional factor process such as
nowadays popular regime-switching process. Systems with memory (time delays)
form another class of important problems. Much work can also be devoted to numerical
solutions of various stochastic differential equations, coordination of multi-agent systems,
and many Monte Carlo optimization problems in which
one needs to treat path-dependent functionals.

\bibliographystyle{plain}
\def\cprime{$'$}

\end{document}